\documentclass{amsart}

\usepackage{amsmath}
\usepackage{amsfonts}
\usepackage{rotating}
\usepackage{bbm}
\usepackage[all,cmtip]{xy}

\newtheorem{fed}{\textbf{Definition}}[section]
\newtheorem{thm}[fed]{\textbf{Theorem}}
\newtheorem{lemma}[fed]{\textbf{Lemma}}

\newtheorem{rem}[fed]{\textbf{Remark}}
\newtheorem{prop}[fed]{\textbf{Proposition}}
\newtheorem{cor}[fed]{\textbf{Corollary}}
\usepackage{amssymb,bbm,graphicx,epsfig,psfrag,epic,eepic,latexsym}
\usepackage{amsmath}
\usepackage{mathrsfs}

\begin{document}
\title{The Omega limit set of a family of chords}
\author{Edward Belbruno, Urs Frauenfelder, Otto van Koert}
\begin{abstract}
In this paper we study the limit behavior of a family of chords on compact energy hypersurfaces of a family of Hamiltonians. Under the assumption that the energy hypersurfaces are all of contact type, we give results on the Omega limit set of this family of chords. Roughly speaking, such a family must either end in a degeneracy, in which case it joins another family, or can be continued.

This gives a Floer theoretic explanation of the behavior of certain families of symmetric periodic orbits in many well-known problems, including the restricted three-body problem. 
\end{abstract}

\maketitle

\section{Introduction}
In many well-known dynamical systems, such as the restricted three-body problem, families of periodic orbits are known to exist. Often existence is proved near an integrable or otherwise easily understood case. How far the family then extends is hard to quantify by analytical means, although one can often verify numerically the existence of such a family for a wide range of parameters.
In this paper we focus our attention on symmetric periodic orbits and more generally chords.

The goal  is to provide a Floer-theoretic reason for the behavior of families of chords. 
The basic setup consists of a symplectic manifold $(M,\omega)$, a smooth $1$-parameter family of autonomous Hamiltonians $H_\mu$ and a pair of exact Lagrangians $L_0$ and $L_1$.
We will assume that $\Sigma_\mu^{-1}:=H_\mu^{-1}(0)$ is a compact hypersurface that is of contact-type.

Suppose that $\{ v_\mu \}_{\mu \in [0,\mu_\infty)}$ is a smooth $1$-parameter family of non-degenerate Reeb chords in $\Sigma_\mu=H^{-1}_\mu(0)$ connecting $L_0$ and $L_1$.
Then one of the following options must hold
\begin{enumerate}
\item the family extends across $\mu_\infty$ to a family $[0,\mu_\infty+\delta)$
\item $v_{\mu_\infty}$ exists and is a degenerate Reeb chord. In this case, there is another family with the same $\Omega$-limit set.
\end{enumerate}
Without the contact condition, the family can cease to exist for other reasons than degeneracy, for example a blue sky catastrophe can occur, meaning that the period blows up as $\mu \to \mu_\infty$. This happens, for example, on the cotangent bundle of a genus $g$ surface.
On this symplectic manifold we can choose a $1$-parameter family of Hamiltonians with $\Sigma_\mu \cong ST^*S_g$, where the dynamics change from geodesic flow into the horocycle flow. The latter has no periodic orbits, showing that the family can simply stop in such a case without ending up in either option (1) or (2).

To give a detailed and general statement, we introduce the $\Omega$-limit set of a family of chords. This consists of all limits of the family when a sequence
of parameters $\mu_\nu$ converges to $\mu_\infty$. Our first main result is
\\ \\
\textbf{Theorem\,A:} \emph{The $\Omega$-limit set is nonempty, compact and connected.}
\\ \\
We state Theorem\,A again in Section~\ref{omega}, where we prove it. The proof relies on the Theorem of Arzela-Ascoli.
Because our family of hypersurfaces is compact, the image of the chords lies in a compact set. The contact condition is
used to prove equicontinuity. For that purpose we interpret chords as critical points of the Rabinowitz action functional. 
\\ \\
Our second main result is
\\ \\
\textbf{Theorem\,B:} \emph{If the $\Omega$-limit set is isolated and the family cannot be extended over the limit set for
$\mu>\mu_\infty$, another family for $\mu<\mu_\infty$ converges to the limit set.}
\\ \\
The precise statement of Theorem\,B together with its proof can be found in Section~\ref{gradient}. Intuitively the
Theorem is rather clear. Because the family cannot be extended over the $\Omega$-limit set the local Rabinowitz Floer homology of the limit set vanishes and therefore for $\mu<\mu_\infty$ there has to be a second family which 
kills the first one. In our proof we do not actually need the full strength of local Rabinowitz Floer homology but use
a homotopy of homotopies argument for gradient flow lines. This is technically easier because it does not involve gluing. 

Similar results for periodic orbits instead of chords should hold true. The advantage for chords is that if one interprets
them as critical points of the Rabinowitz action functional they are generically Morse critical points. This never happens for
periodic orbits because as critical points of the Rabinowitz action functional they are parametrized and by reparametrizing them one gets different critical points of the Rabinowitz action functional. In particular, critical points are never isolated. 
One can interpret reparametrization as a circle action on the free loop space and in the periodic orbit case the Rabinowitz action functional is invariant under this circle action. Therefore it is generically Morse-Bott and its critical points arise in circle families. It should be interesting to consider the local equivariant Rabinowitz Floer homology of an $\Omega$-limit
set of a family of periodic orbits to understand what happens if the family cannot be extended over the limit set.   

To conclude the introduction, let us point out that many classical dynamical systems have been proved to be of contact-type. For example, regular energy levels of all mechanical Hamiltonians as well as regularized energy hypersurfaces of the planar and spatial restricted three-body problem in a large range of energy values, see \cite{albers-frauenfelder-koert-paternain,cho-kim}, are of contact-type.
Mechanical Hamiltonians with compact energy hypersurfaces include the H\'enon-Heiles Hamiltonian, see \cite{salomao} for a discussion of some of its properties related to convexity, which is stronger than the contact condition.

The results of this paper hence apply to these systems. We have included numerics illustrating some of these phenomena of families in the restricted three-body problem in Figure~\ref{fig:planar_RTBP}. Results of this paper were previously announced and applied to the spatial restricted three-body problem in our earlier paper \cite{belbruno-frauenfelder-vankoert} on polar orbits in the lunar problem.

\subsection*{Acknowledgements}
Edward Belbruno would like acknowledge the support of Humboldt Stiftung of the Federal Republic of Germany that made this research possible and the support of the University of Augsburg for his visit from 2018 until 2019.
Urs Frauenfelder was supported by DFG grant FR 2637/2-1 of the German government.
Otto van Koert was supported by NRF grant NRF-2016R1C1B2007662, which was funded by the Korean Government.

\section{Chords}
Suppose that $(M,\lambda)$ is an exact symplectic manifold, i.e., $\lambda \in \Omega^1(M)$ is a one-form such that 
$$
\omega=d\lambda
$$ 
is a symplectic form. Assume also that $L_0, L_1 \subset M$ are two exact Lagrangian submanifolds
in the sense that they are Lagrangian submanifolds of $(M,\omega)$ with the additional property that 
$$
\lambda|_{L_i}=0
$$
for $i \in \{0,1\}$. Furthermore, we are given a smooth function
$$
H \colon M \to \mathbb{R},
$$ 
referred to as the Hamiltonian, with the property that $0$ is a regular value of $H$ so that its level set
$$\Sigma=H^{-1}(0)$$
is a smooth hypersurface in $M$. We assume further that
$$\Sigma \pitchfork L_i, \quad i \in {0,1},$$
i.e. $\Sigma$ intersects $L_i$ transversally in the sense that if $x \in \Sigma \cap L_i$ then
$$T_x M=T_x \Sigma+T_x L_i.$$
In particular, this implies that
$$\mathcal{L}_i:=L_i \cap \Sigma, \quad i \in \{0,1\}$$
are smooth submanifolds of $\Sigma$. If the dimension of $M$ is $2n$, then the dimension of $\Sigma$ is $2n-1$ and the dimension of each
$\mathcal{L}_i$ is $n-1$. The Hamiltonian vector field of $H$ is implicitly defined by the condition
$$dH=\omega(\cdot,X_H).$$
Note that by antisymmetry of the symplectic form we get
$$dH(X_H)=\omega(X_H,X_H)=0$$
so that $X_H$ is tangent to the energy hypersurface $\Sigma$. In particular, $\Sigma$ is invariant under the flow of $X_H$. If one thinks of $H$ as energy then this means that energy is preserved. 

\begin{lemma}\label{nottan}
The Hamiltonian vector field $X_H$ is never tangent to $\mathcal{L}_i$ for $i \in \{0,1\}$.
\end{lemma}
\begin{proof}
This is a consequence of the assumption that $L_i$ is transverse to $\Sigma$. Indeed, suppose that $x \in \mathcal{L}_i$. Because the symplectic form is non-degenerate there exists
$\eta \in T_x M$ such that 
$$\omega(X_H,\eta) \neq 0.$$
Since $L_i$ is transverse to $\Sigma$ we can decompose
$$\eta=\eta_0+\eta_1, \quad \eta_0 \in T_x L_i,\,\,\eta_1 \in T_x \Sigma.$$
Hence
$$0 \neq \omega(X_H,\eta)=\omega(X_H,\eta_0)+\omega(X_H,\eta_1)=\omega(X_H,\eta_0)+dH(\eta_1)=\omega(X_H,\eta_0).$$
Here we have used for the last equation that $\eta_1$ is tangent to the level set of $H$. Since $\eta_0$ is a tangent vector of the Lagrangian, it follows from the definition of a Lagrangian subspace that
$$X_H \notin T_x L_i.$$
In particular, $X_H$ is never tangent to $\mathcal{L}_i$. This finishes the proof of the lemma. 
\end{proof} 

\begin{fed}
A \emph{chord} $(v,\tau) \in C^\infty([0,1],\Sigma) \times (0,\infty)$ from $\mathcal{L}_0$ to $\mathcal{L}_1$ is a solution of the problem
$$\left\{\begin{array}{cc}
\partial_t v(t)=\tau X_H(v(t)), & t \in [0,1],\\
v(i) \in \mathcal{L}_i, & i \in \{0,1\}.
\end{array}\right.$$
\end{fed}
If we reparametrize a chord $(v,\tau)$ to
$$v_\tau(t):=v\big(\tfrac{t}{\tau}\big), \quad t \in [0,\tau]$$
then $v_\tau \in C^\infty([0,\tau],\Sigma)$ is a solution of the problem
$$\left\{\begin{array}{cc}
\partial_t v_\tau(t)=X_H(v_\tau(t)), & t \in [0,\tau],\\
v_\tau(0) \in \mathcal{L}_0, & \\
v_\tau(\tau) \in \mathcal{L}_1. &
\end{array}\right.$$
In view of this reparametrization we refer to $\tau$ as the \emph{period} of the chord. 
\\ \\
We abbreviate by $\phi^t_H$ the flow of the Hamiltonian vector field $X_H$, i.e.
$$\phi^0_H=\mathrm{id}|_M,\quad \frac{d}{dt}\phi^t_H(x)=X_H(\phi^t_H(x)),\quad x \in M.$$
If we set
$$\Phi \colon \mathcal{L}_0 \times (0,\infty) \to \Sigma, \quad (x,\tau) \mapsto \phi_H^\tau(x)$$
then the map
$$(v,\tau) \mapsto (v(0),\tau)$$
gives a one to one correspondence between the set of chords and the set $\Phi^{-1}(\mathcal{L}_1)$.
In the following definition we use this identification. 
\begin{fed}
A chord $(v,\tau)$ is called \emph{non-degenerate} if $\Phi$ is transverse to $\mathcal{L}_1$ at $(v,\tau)$, i.e.,
$$d \Phi(v,\tau) T_{(v,\tau)}(\mathcal{L}_0 \times (0,\infty))\oplus T_{\Phi(v,\tau)}\mathcal{L}_1=
T_{\Phi(v,\tau)}\Sigma.$$
Otherwise, the chord is called \emph{degenerate}.
\end{fed}
In view of Lemma~\ref{nottan} we get the following equivalent characterization of a non-degenerate chord.
\begin{lemma}
A chord $(v,\tau)$ is non-degenerate if and only if
$$d\phi^\tau(v(0)) T_{v(0)}\mathcal{L}_0 \cap T_{v(1)}\mathcal{L}_1=\{0\}.$$
\end{lemma}
\section{Rabinowitz action functional}
We will see that chords can be detected variationally as critical points of the Rabinowitz action functional.
For periodic orbits this functional was first considered in~\cite{rabinowitz},~Equation 2.7, and a Floer theory for this functional was developed in~\cite{cieliebak-frauenfelder}.
Extensions to the chord case can be found in \cite{kang,merry}. Some of the arguments that will be used here were earlier explored in \cite{albers-frauenfelder}.

Abbreviate by
$$\mathcal{P}=\big\{v \in C^\infty([0,1],M): v(i) \in L_i,\,\,i \in \{0,1\}\big\}$$
the space of paths in $M$ connecting $L_0$ with $L_1$. 
The Rabinowitz action functional
$$\mathcal{A}^H \colon \mathcal{P} \times (0,\infty) \to \mathbb{R}$$
at a point $(v,\tau) \in \mathcal{P} \times (0,\infty)$ is given by
$$\mathcal{A}^H(v,\tau)=\int_0^1 v^* \lambda-\tau \int_0^1 H(v(t)) dt.$$
The first term is just the area functional. One might think of $\tau$ as a Lagrange multiplier. Then the critical points of the Rabinowitz action functional correspond to critical points of the area functional subject to the constraint that the mean value of
the Hamiltonian $H$ has to vanish.
\begin{prop}\label{crit}
Critical points of $\mathcal{A}^H$ correspond to chords.
\end{prop}
\begin{proof}
For $v \in \mathcal{P}$, the tangent space of $\mathcal{P}$ at $v$ 
$$T_v \mathcal{P}=\big\{\widehat{v} \in \Gamma(v^* TM): \widehat{v}(i) \in T_{v(i)}L_i,\,\,i \in \{0,1\}\big\}$$
consists of vector fields along $v$ starting and ending in the corresponding Lagrangians. We first consider the differential
of the area functional
$$\mathcal{A}_0 \colon \mathcal{P} \to \mathbb{R}, \quad v \mapsto \int_0^1 v^* \lambda.$$
If $\widehat{v} \in T_v \mathcal{P}$ and $L_{\widehat{v}}$ denotes the Lie derivative in the direction of $\widehat{v}$, then we can compute using Cartan's formula
\begin{eqnarray*}
d \mathcal{A}_0(v) \widehat{v}&=&\int_0^1v^* L_{\widehat{v}}\lambda\\
&=&\int_0^1 v^*d \iota_{\widehat{v}}\lambda+\int_0^1 v^*\iota_ {\hat{v}}d\lambda\\
&=&\int_0^1 d v^* \iota_{\widehat{v}}\lambda+\int_0^1 v^*\iota_ {\hat{v}}\omega\\
&=&\lambda(v(1))\widehat{v}(1)-\lambda(v(0))\widehat{v}(0)+\int_0^1 \omega(\widehat{v}, \partial_t v)dt\\
&=&\int_0^1 \omega(\widehat{v}, \partial_t v)dt.
\end{eqnarray*}
Here we have used in the fourth equality Stokes' theorem and in the last equality we have taken advantage of the assumption that
$\lambda$ vanishes on $L_0$ and $L_1$. Using this formula we are now in position to compute the differential of the Rabinowitz
action functional. At a point $(v,\tau) \in \mathcal{P}\times (0,\infty)$ and a tangent vector
$$(\widehat{v},\widehat{\tau}) \in T_{(v,\tau)}\big(\mathcal{P}\times (0,\infty)\big)=T_v\mathcal{P} \times \mathbb{R}$$ 
we get
\begin{eqnarray}\label{rabdiff}
d \mathcal{A}^H(v,\tau)(\widehat{v},\widehat{\tau})&=&d\mathcal{A}_0(v)\widehat{v}-\tau \int_0^1 dH(v)\widehat{v}dt-
\widehat{\tau}\int_0^1 H(v)dt\\ \nonumber
&=&\int_0^1 \omega\big(\widehat{v}, \partial_t v-\tau X_H(v)\big)dt-\widehat{\tau}\int_0^1 H(v)dt.
\end{eqnarray}
At a critical point $(v,\tau)$ of the Rabinowitz action functional this expression has to vanish for all $(\widehat{v},\widehat{\tau}) \in T_v \mathcal{P} \times \mathbb{R}$ implying that $(v,\tau)$ solves
$$\left\{\begin{array}{cc}
\partial_t v(t)=\tau X_H(v(t)), & t \in [0,1],\\
\int_0^1 H(v)dt=0. &
\end{array}\right.$$
Because the Hamiltonian vector field is tangent to the level sets of the Hamiltonian we obtain from the first equation that
$H(v)$ is constant. Hence the mean value constraint in the second equation is actually equivalent to a pointwise constraint and the equation above can be equivalently written as
$$\left\{\begin{array}{cc}
\partial_t v(t)=\tau X_H(v(t)), & t \in [0,1],\\
H(v(t))=0, & t \in [0,1].
\end{array}\right.$$
But this is precisely the equation of a chord. This finishes the proof of the proposition. 
\end{proof} 

\begin{rem}
Under the identification of chords with critical points of the Rabinowitz action functional, the non-degeneracy condition of a chord  can be thought of as a Morse condition for the critical point. 
\end{rem}

\section{The contact condition}

On the exact symplectic manifold $(M,\omega=d\lambda)$ the Liouville vector field $Y$ is defined implicitly by the condition
$$\lambda=\omega(Y,\cdot).$$
We now assume the following contact condition on $\Sigma$
$$dH|_\Sigma(Y)>0.$$
This implies that
$$Y \pitchfork \Sigma,$$
i.e., the Liouville vector field is transverse to the energy hypersurface $\Sigma$ and the restriction of the one-form $\lambda$
to $\Sigma$ is a contact form. Under this assumption we can define the 
Reeb vector field on $\Sigma$ uniquely by the requirement
$$\lambda(R)=\omega(Y,R)=1,\quad \omega(R,\xi)=0,\,\,\xi \in T\Sigma.$$
Note that the restriction of the Hamiltonian vector field to $\Sigma$ is positively parallel to the Reeb vector field. Indeed, if we define  
$$f \colon \Sigma \to (0,\infty), \quad x \mapsto dH(Y)(x)$$ it follows that
$$\omega(Y,X_H)|_\Sigma=dH(Y)|_\Sigma=f$$
which together with 
$$0=dH (\xi)=\omega(\xi,X_H)=-\omega(X_H,\xi), \quad \xi \in T\Sigma$$
implies that for every $x \in \Sigma$ it holds that
$$X_H(x)=f(x)R(x).$$
We suppose further that $\Sigma$ is compact. As a consequence there exists $\kappa\geq 1$ such that
\begin{equation}\label{est}
\frac{1}{\kappa} \leq f(x) \leq \kappa,\quad x \in \Sigma.
\end{equation}
\begin{lemma}\label{esti}
Under the above assumptions, if $(v,\tau)$ is a chord, then its action can be estimated by the period as
$$\frac{\tau}{\kappa}\leq \mathcal{A}^H(v,\tau) \leq \kappa \tau.$$
\end{lemma}
\begin{proof}
We compute
$$\mathcal{A}^H(v,\tau)=\mathcal{A}_0(v)=\int_0^1\lambda(\tau X_H(v(t))dt=\tau \int_0^1 f(v(t)) dt.$$
The estimate now follows in view of \eqref{est}. \end{proof}

\section{The Omega limit set}\label{omega}

We now suppose that we have a one-parameter family of Hamiltonian functions, namely a smooth function $H \colon M \times [0,1] \to \mathbb{R}$. For $\mu \in [0,1]$ we abbreviate
$$H_\mu:=H(\cdot,\mu) \in C^\infty(M).$$
Suppose that $0$ is a regular value of $H_\mu$ for every $\mu \in [0,1]$ and $H^{-1}(0)$ is compact. This implies that 
the level sets
$$\Sigma_\mu:=H_\mu^{-1}(0)$$
build a smooth family of compact hypersurfaces in $M$ and in particular all $\Sigma_\mu$ are diffeomorphic to each other. 
If $Y$ is the Liouville vector field on $M$ we assume that 
$$dH_\mu|_{\Sigma_\mu}(Y)>0$$
so that for every $\mu \in [0,1]$ the energy hypersurface $\Sigma_\mu$ satisfies the contact condition. 
If $R_\mu \in \Gamma(T \Sigma_\mu)$ denotes the Reeb vector field on $\Sigma_\mu$, then with the smooth function
$$f \colon H^{-1}(0) \to (0,\infty),\quad (x,\mu) \mapsto dH_\mu(Y)(x)$$ it holds for every $\mu \in [0,1]$ and for every $x \in \Sigma_\mu$
that
$$X_{H_\mu}(x)=f(x,\mu)R_\mu(x).$$
Abbreviate
$$H'_\mu:=\frac{\partial H(\cdot,\mu)}{\partial\mu} \in C^\infty(M).$$
By compactness there exists $\kappa\geq 1$ with the property that 
\begin{equation}\label{supest}
\frac{1}{\kappa} \leq f(x,\mu) \leq \kappa,\quad |H'_\mu(x)| \leq \kappa,\quad (x,\mu) \in H^{-1}(0).
\end{equation}
Suppose that $(v_0,\tau_0)$ is a non-degenerate chord on $\Sigma_0$. Due to the assumption that the chord is non-degenerate, there exists by the implicit function theorem $\mu_\infty=\mu_\infty(v_0,\tau_0) \in (0,1]$ with the property that there exist smooth maps
$$v \colon [0,1]\times[0,\mu_\infty) \to M, \quad \tau \colon [0,\mu_\infty) \to (0,\infty)$$
with the property that $(v(0),\tau(0))=(v_0,\tau_0)$ and for each $\mu \in [0,\,\mu_\infty)$ the tuple
$$(v_\mu,\tau_\mu):=(v(\cdot,\mu),\tau(\mu)) \in C^\infty([0,1],M) \times (0,\infty)$$
is a non-degenerate chord on $\Sigma_\mu$. We assume that $\mu_\infty$ is maximal with this property.
\begin{prop}\label{bound}
There exist constants $0<c_0<c_1<\infty$ such that
$$c_0\leq \tau_\mu \leq c_1, \quad \mu \in [0,\mu_\infty).$$
\end{prop}

\begin{proof}
The method of proof is very similar to the argument in \cite{frauenfelder-vankoert}, Theorem~7.6.1, used to preclude blue sky catastrophes.
By Proposition~\ref{crit} for every $\mu \in [0,\mu_\infty)$ we can interpret the chord $(v_\mu,\tau_\mu)$ as 
a critical point of the Rabinowitz action functional $\mathcal{A}^{H_\mu}$. Differentiating the action along the family of critical points we obtain
$$\frac{d}{d\mu}\mathcal{A}^{H_\mu}(v_\mu,\tau_\mu)=\mathcal{A}^{H'_\mu}(v_\mu,\tau_\mu)=-\tau_\mu \int_0^1 H_\mu'(v_\mu)dt.$$
In view of Lemma~\ref{esti} and the inequalities \eqref{supest} we estimate from this
\begin{equation}\label{ineq}
\frac{1}{\kappa^2}\mathcal{A}^{H_\mu}(v_\mu,\tau_\mu) \leq \bigg|\frac{d}{d\mu}\mathcal{A}^{H_\mu}(v_\mu,\tau_\mu)\bigg|
\leq \kappa^2 \mathcal{A}^{H_\mu}(v_\mu,\tau_\mu).
\end{equation}
Integrating this inequality we obtain the inequality
$$e^{-\kappa^2\mu}\mathcal{A}^{H_0}(v_0,\tau_0)\leq \mathcal{A}^{H_\mu}(v_\mu,\tau_\mu) \leq e^{\kappa^2\mu}\mathcal{A}^{H_0}(v_0,\tau_0).$$
Using once more Lemma~\ref{esti} we obtain from that the estimate
$$\frac{e^{-\kappa^2 \mu}}{\kappa}\tau_0 \leq \tau_\mu \leq \kappa e^{\kappa^2 \mu} \tau_0.$$
In particular, $\tau_\mu$ is uniformly bounded from above and below by
$$\frac{e^{-\kappa^2}}{\kappa}\tau_0 \leq \tau_\mu \leq \kappa e^{\kappa^2} \tau_0.$$
This finishes the proof of the proposition. 
\end{proof}

Let 
$$\Omega=\Omega(v_0,\tau_0)\subset \mathcal{P}\times (0,\infty)$$
be the Omega limit set of the family of chords $(v,\tau)$ consisting of all $(w,\sigma) \in \mathcal{P}\times (0,\infty)$
for which there exists a sequence $\mu_\nu \in [0,\mu_\infty)$ for $\nu \in \mathbb{N}$ such that 
$$\lim_{\nu \to \infty}
\mu_\nu=\mu_\infty, \quad \lim_{\nu \to \infty}(v_{\mu_\nu},\tau_{\mu_\nu})=(w,\sigma).$$
Here it suffices to require that the second limit is in the $C^0$-topology, since by bootstrapping the equation of a chord
we automatically get that the limit is actually in the $C^\infty$-topology and in particular, the Omega limit set consists of
chords on $\Sigma_{\mu_\infty}$.
\begin{thm}
The Omega limit set $\Omega$ is nonempty, compact and connected. Moreover, the Rabinowitz action functional
$\mathcal{A}^{H_{\mu_\infty}}$ is constant on $\Omega$. If $\mu_\infty<1$, then all chords in $\Omega$ are degenerate. If
$\mu_\infty=1$ and $\Omega$ consists of more than just one chord, then again all chords in $\Omega$ are degenerate.  
\end{thm}
\begin{proof}
Suppose that $\mu_\nu \in [0,\mu_\infty)$ is a sequence converging to $\mu_\infty$. By Proposition~\ref{bound}
the periods $\tau_{\mu_\nu}$ are uniformly bounded from above as well as uniformly bounded from below away from zero. By the chord equation the sequence $v_{\mu_\nu}$ is equicontinuous so that in view of the assumption that $H^{-1}(0)$ is compact 
we obtain by the Theorem of Arzela-Ascoli a convergent subsequence of $(v_{\mu_\nu},\tau_{\mu_\nu})$. In particular, $\Omega$
is not empty.

We next show that $\Omega$ is compact. Because the path space is metrizable, it suffices to show that $\Omega$ is sequentially
compact. Assume that 
$$x_\nu=(w_\nu,\sigma_\nu)$$ 
is a sequence in $\Omega$. By definition of $\Omega$ for each $\nu \in \mathbb{N}$ there exists a sequence $\mu_k^\nu$ converging to $\mu_\infty$ as $k$ goes to infinity such that the sequence
$$y_k^\nu:=(v_{\mu_k^\nu},\tau_{\mu_k^\nu})$$
converges to $x_\nu$ as $k$ goes to infinity. Choose a metric on $\mathcal{P}\times (0,\infty)$ which induces the topology.
We put $k_1=1$ and inductively define for $\nu \in \mathbb{N}$
$$k_{\nu+1}:=\min\bigg\{k:\mu_\infty-\mu_k^{\nu+1}\leq \frac{\mu_\infty-\mu_{k_\nu}^\nu}{2},
d(y_k^{\nu+1},\,\,x_{\nu+1}) \leq \frac{1}{\nu+1}\bigg\}.$$ 
We set
$$\mu_\nu:=\mu^\nu_{k_\nu}.$$
From the construction of $k_\nu$ we infer that the sequence $\mu_\nu$ converges to $\mu_\infty$ as $\nu$ goes to infinity. By the argument in the first paragraph of this proof it follows that there exists a subsequence $\nu_j$ and a chord
$x=(w,\sigma) \in \Omega$ such that 
$$\lim_{j \to \infty}(v_{\mu_{\nu_j}},\tau_{\mu_{\nu_j}})=x.$$
In order to prove compactness it suffices now to show that
$$\lim_{j \to \infty}x_{\nu_j}=x.$$
Noting that
$$y^{\nu_j}_{k_{\nu_j}}=(v_{\mu_{\nu_j}},\tau_{\mu_{\nu_j}})$$
we choose for $\epsilon>0$ a positive integer $j_0=j_0(\epsilon)$ satisfying
$$
\nu_{j_0} \geq \frac{2}{\epsilon},\qquad d\big(y^{\nu_j}_{k_{\nu_j}},x\big)\leq \frac{\epsilon}{2}, \quad \forall\,\,j \geq j_0.
$$
For $j \geq j_0$ we get the estimate
$$
d(x_{\nu_j},x)\leq d\big(x_{\nu_j},y^{\nu_j}_{k_{\nu_j}}\big)+d\big(y^{\nu_j}_{k_{\nu_j}},x\big)
\leq \frac{1}{\nu_j}+\frac{\epsilon}{2}\leq \frac{1}{\nu_{j_0}}+\frac{\epsilon}{2}=\epsilon.
$$ 
We have proved that $\Omega$ is compact.

The goal of this paragraph is to show that $\Omega$ is connected. We argue by contradiction and assume that $\Omega$ is not connected. This means that $\Omega$ can be written as
$$\Omega=\Omega_1 \cup \Omega_2$$
where both $\Omega_1$ and $\Omega_2$ are nonempty, closed and open subsets of $\Omega$ which are disjoint from each other, i.e.,
$$\Omega_1 \cap \Omega_2=\emptyset.$$
As we have already proved that $\Omega$ is compact, the sets $\Omega_1$ and $\Omega_2$ are compact as well. This allows us to find disjoint open neighborhoods  of $\Omega_1$ and $\Omega_2$, i.e., open subsets $U_1$ and $U_2$ in $\mathcal{P}\times (0,\infty)$ satisfying
$$U_1 \cap U_2=\emptyset, \quad \Omega_1 \subset U_1,\quad \Omega_2 \subset U_2.$$
Choose
$$x_1=(w_1,\sigma_1) \in \Omega_1, \quad x_2=(w_2,\sigma_2) \in \Omega_2.$$
Because both sets $\Omega_1$ and $\Omega_2$ are nonempty this is possible. Moreover, by construction of $\Omega$ there
exist sequences $\mu_k^1$ and $\mu_k^2$ in $[0,\mu_\infty)$ converging to $\mu_\infty$ such that
$$\lim_{k\to \infty}(v_{\mu_k^i},\tau_{\mu_k^i})=x_i, \quad i \in \{1,2\}.$$
We put $k_1=1$ and inductively define for $\nu \in \mathbb{N}$
$$k_{\nu+1}:=\left\{\begin{array}{cc}
\min\big\{k:\mu_k^2>\mu^1_{k_\nu}\big\} & \nu\,\,\mathrm{odd}\\
\min\big\{k:\mu_k^1>\mu^2_{k_\nu}\big\} & \nu\,\,\mathrm{even}.
\end{array}\right.$$
By construction the sequence $\mu_{k_\nu}$ converges to $\mu_\infty$. 
For $\nu$ large enough the path
$$
[\mu_{k_\nu},\mu_{k_{\nu+1}}]\to\mathcal{P}\times (0,\infty),\quad
\mu \mapsto (v_\mu,\tau_\mu)
$$
has the property that one of the endpoints of this path lies in $U_1$, whereas the
other is contained in $U_2$. Since $U_1$ and $U_2$ are disjoint, this allows us
to find 
$$\mu_\nu \in [\mu_{k_\nu},\mu_{k_{\nu+1}}]$$
meeting the requirement that
$$y_\nu:=(v_{\mu_\nu},\tau_{\mu_\nu}) \in \big(\mathcal{P}\times (0,\infty)\big)
\setminus (U_1 \cup U_2).$$
By construction the sequence $\mu_\nu$ converges to $\mu_\infty$ as $\nu$ goes to
infinity. Therefore by the first paragraph of this proof we conclude that there
exists a subsequence $\nu_j$ and 
$$x \in \Omega$$ 
such that 
$$\lim_{j \to \infty} y_{\nu_j}=x.$$
However, $U_1$ and $U_2$ were assumed to be open and therefore their complement is closed so that 
$$y \in \big(\mathcal{P} \times (0,\infty)\big) \setminus (U_1 \cap U_2)
\subset \big(\mathcal{P} \times (0,\infty)\big) \setminus (\Omega_1 \cap \Omega_2)
\subset \big(\mathcal{P} \times (0,\infty)\big) \setminus \Omega$$
This is a contradiction and therefore $\Omega$ has to be connected. 

We next show that the Rabinowitz action functional $\mathcal{A}^{H_{\mu_\infty}}$ is constant
on $\Omega$. By integrating \eqref{ineq} we obtain for every $0\leq \mu_1<\mu_2<\mu_\infty$ the
inequality
$$e^{-\kappa^2(\mu_2-\mu_1)}\mathcal{A}^{H_{\mu_1}}(v_{\mu_1},\tau_{\mu_1}) \leq
\mathcal{A}^{H_{\mu_2}}(v_{\mu_2},\tau_{\mu_2}) \leq 
e^{\kappa^2(\mu_2-\mu_1)}\mathcal{A}^{H_{\mu_1}}(v_{\mu_1},\tau_{\mu_1})$$
from which the desired conclusion follows. 

Finally if there is a non-degenerate chord in $\Omega$, then in view of the fact that $\Omega$
is connected, as proved above, it follows from the implicit function theorem, that $\Omega$ just
consists of one single chord and the family $(v,\tau)$ can be extended to $\mu>\mu_\infty$
unless $\mu_\infty=1$. This completes the proof of the theorem. 
\end{proof}

\section{Gradient flow lines}\label{gradient}

We assume that the same set-up as in Section~\ref{omega} holds. We denote by $\mathcal{C}_\mu \subset \mathcal{P}
\times (0,\infty)$ the set of all chords on $\Sigma_\mu$ for $\mu \in [0,1]$ or equivalently
$$\mathcal{C}_\mu=\mathrm{crit}\mathcal{A}^{H_\mu}$$
the set of critical points of the Rabinowitz action functional. We suppose that the following hypotheses hold true
\begin{description}
 \item[(i)] The Omega limit set $\Omega$ is isolated in the set of chords on $\Sigma_{\mu_\infty}$, i.e., there exists
 an open set $U \subset \mathcal{P} \times (0,\infty)$ such that
 $$U \cap \mathcal{C}_{\mu_\infty}=\Omega.$$
  \item[(ii)] The family of chords $(v,\tau)$ does not extend to $1$, i.e.,
  $$\mu_\infty<1$$ 
  and there exists a sequence $\mu_\nu \in (\mu_\infty,1]$ converging to $\mu_\infty$ with the
  property that on $\Sigma_{\mu_\nu}$ there are no chords in $U$, i.e., 
  $$\mathcal{C}_{\mu_\nu} \cap U=\emptyset, \quad \nu \in \mathbb{N}.$$
\end{description}
\begin{thm}\label{mainthm2}
Under the assumptions above there exists $\delta>0$ with the property that for all $\mu \in (\mu_\infty-\delta,\mu_\infty)$
there are on $\Sigma_\mu$ at least two chords in $U$, i.e.,
$$\#(\mathcal{C}_\mu \cap U) \geq 2.$$
\end{thm}
By definition of the Omega limit set we already know that for $\mu \in [0,\mu_\infty)$ close enough to $\mu_\infty$ the chord
$(v_\mu,\tau_\mu)$ belongs to $\mathcal{C}_\mu \cap U$. The theorem guarantees that there is a second chord on $\Sigma_\mu$ for $U$. 
This is clear intuitively by looking at the local Rabinowitz Floer homology for the isolated critical set $\Omega$. By hypothesis~(ii), the local Rabinowitz Floer homology is trivial. Now if $(v_\mu,\tau_\mu)$ were the only chord on $\Sigma_\mu$ for $U$, then the local Rabinowitz Floer homology is a one-dimensional vector space generated by
$(v_\mu,\tau_\mu)$, because $(v_\mu,\tau_\mu)$ is by definition non-degenerate.
From this contradiction we see that there have to be at least two elements in $\mathcal{C}_\mu \cap U$.
In the following we give a more elementary argument to prove Theorem~\ref{mainthm2} which does not use the full strength of local Rabinowitz Floer homology. In particular, it does not require any gluing construction, but is based on Floer's stretching method for time dependent gradient flow lines. 

\begin{figure}[htp]
\def\svgwidth{0.7\textwidth}%
\begingroup\endlinechar=-1
\resizebox{0.7\textwidth}{!}{%
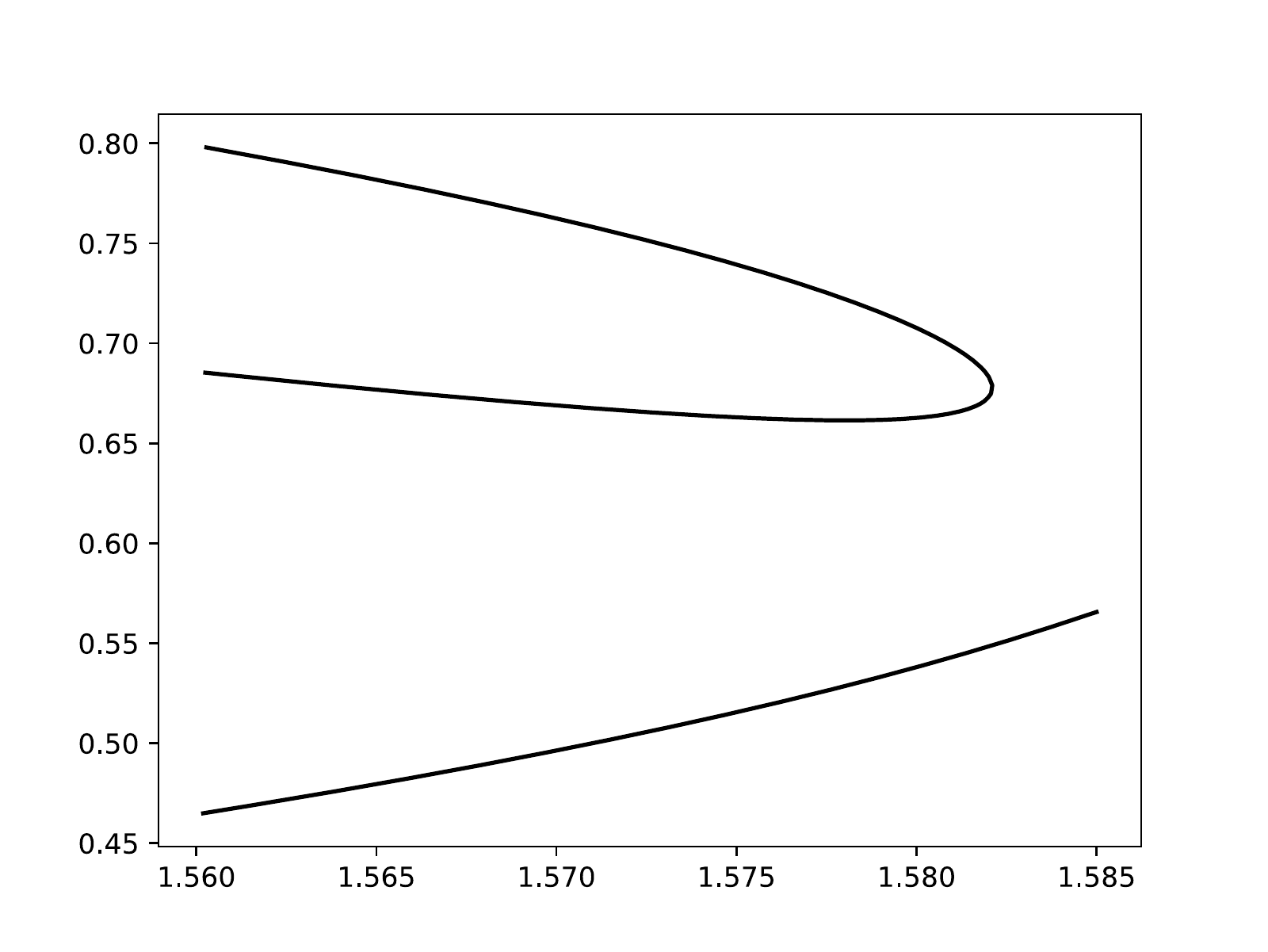%
}\endgroup
\caption{Two families of symmetric periodic orbits in RTBP for $\mu=10^{-3}$.
The family parameter is the Jacobi energy $c$. On the vertical axis, the starting point of the chord, parametrized by the $x$-axis.
The lower family is non-degenerate in the entire parameter range. The upper family becomes degenerate. 
}
  \label{fig:planar_RTBP}
\end{figure}

As preparation for the proof of Theorem~\ref{mainthm2} we choose an $\omega$-compatible almost complex structure $J$ on
$M$, meaning that
$$
g=\omega(\cdot,J\cdot)
$$
is a Riemannian metric on $M$. We use the Riemannian metric $g$ to define a metric on each connected component of $\mathcal{P}$ as follows. 
If $w_0, w_1 \in \mathcal{P}$ let 
$$\mathscr{P}_{w_0,w_1} \subset C^\infty([0,1],[0,1],M)$$
be the subspace consisting of all $w \in C^\infty([0,1],[0,1],M)$ satisfying the boundary conditions
$$ w(0,\cdot)=w_0,\,\,w(1,\cdot)=w_1,\quad 
w(0,s) \in L_0,\,\,w(1,s)\in L_1,\,\,s \in [0,1],$$
i.e., $\mathscr{P}_{w_0,w_1}$ consists of all paths in the pathspace $\mathcal{P}$ connecting $w_0$ and $w_1$.
If $\mathcal{P}_{w_0,w_1} \neq \emptyset$, i.e., $w_0$ and $w_1$ lie in the same connected component of $\mathcal{P}$, we define
$$
d_{\mathcal{P}}(w_0,w_1):=\inf_{w \in \mathscr{P}_{w_0,w_1}}\int_0^1 \sqrt{\int_0^1 |\partial_s w(s,t)|^2 dt} ds
=\inf_{w \in \mathscr{P}_{w_0,w_1}}\int_0^1||\partial_s w(s,\cdot)||_2 ds,
$$
where the norm $|\cdot|=|\cdot|_g$ on $TM$ is taken with respect to the metric $g$ and $||\cdot||_2$ is the $L^2$-norm
on $T \mathcal{P}$ induced from $|\cdot|_g$. This defines a metric on each connected component of $\mathcal{P}$. 
We endow each connected component of $\mathcal{P} \times (0,\infty)$ with the product metric
$$
d\big((w_0,\sigma_0),(w_1,\sigma_1)\big)=d_\mathcal{P}(w_0,w_1)+|\sigma_1-\sigma_0|.
$$
For $\rho>0$ we define the $\rho$-ball around $\Omega$ as
$$B_\rho(\Omega):=\big\{y \in \mathcal{P}\times(0,\infty): d(y,\Omega)<\rho\big\}.$$
\begin{lemma}\label{ball}
There exists $\rho>0$ with the property that
$$B_\rho(\Omega) \cap \mathcal{C}_{\mu_\infty}=\Omega,$$
i.e., the only chords contained in the $\rho$-ball around $\Omega$ are the ones contained in $\Omega$.
\end{lemma}
\begin{proof}
We argue by contradiction and assume that there exists a sequence $\rho_\nu>0$ converging to $0$ with the
property that there exists a chord 
$$y_\nu=(w_\nu,\sigma_\nu) \in B_{\rho_\nu}(\Omega)$$
not contained in $\Omega$. Because $\Omega$ is compact, it follows that the sequence $\sigma_\nu$ is uniformly bounded. By
the chord equation we conclude that the sequence $w_\nu$ is equicontinuous. Hence by the theorem of Arzela-Ascoli there exists
a chord
$$y=(w,\sigma) \in \Omega$$
and a subsequence $\nu_j$ such that
$$\lim_{j \to \infty}y_{\nu_j}=y.$$
Again by the chord equation we conclude that for $j$ large enough $y_{\nu_j} \in U$. This contradicts hypothesis (i)
and the lemma follows. 
\end{proof}

Note that the equation of a chord on $\Sigma_\mu=H_\mu^{-1}(0)$ only depends on the derivatives of $H$ on $\Sigma_\mu$; 
the behavior of $H$ on the rest of the symplectic manifold $M$ is irrelevant for the chord equation. Hence after
maybe throwing away some part of $M$ and maybe after replacing the interval $[0,1]$ by a closed subinterval we can assume
without loss of generality that there exists a constant $c>0$ with the properties that
\begin{equation}\label{bou}
|X_{H_\mu}(x)| \leq c, \quad |H'_\mu(x)|\leq c, \quad |X_{H'_\mu}(x)| \leq c, \quad x \in M
\end{equation}
where we have taken the norm of the tangent vector $|X_H(x)|$ with respect to the metric $g$ and moreover,
$$H^{-1}\big(\big[-\tfrac{1}{c},\tfrac{1}{c}\big]\big) \subset M \times [0,1]$$
is compact. Hence by Lemma~\ref{ball} we can choose $\rho_0>0$ with the following properties.
\begin{description}
 \item[(a)] $B_{\rho_0}(\Omega) \cap \mathcal{C}_{\mu_\infty}=\Omega$, 
 \item[(b)] For every $y=(w, \sigma) \in B_{\rho_0}(\Omega)$ there exists $\tau_0 \in [0,1]$ such that 
 $$|H_{\mu_\infty}(w(\tau_0))| \leq \tfrac{1}{2c},$$
 \item[(c)] $\rho_0 \leq \tfrac{1}{2}\min\big\{\sigma: (w,\sigma) \in \Omega\big\}$.
\end{description}
We next introduce the gradient of the Rabinowitz action functional. In order to do that we first need a Riemannian metric on
$T\big(\mathcal{P}\times (0,\infty)\big)$. Let $y=(w,\sigma) \in \mathcal{P}\times (0,\infty)$ and
$$\widehat{y}_1=(\widehat{w}_1,\widehat{\sigma}_1),\,\,\widehat{y}_2=(\widehat{w}_2,\widehat{\sigma}_2) \in 
T_y\big(\mathcal{P}\times (0,\infty)\big)=T_w \mathcal{P}\times \mathbb{R}.$$
Define
$$\langle \widehat{y}_1,\widehat{y}_2 \rangle:=\int_0^1 g\big(\widehat{w}_1(t),\widehat{w}_2(t)\big)dt+\widehat{\sigma}_1
\cdot \widehat{\sigma}_2=\int_0^1 \omega\big(\widehat{w}_1(t),J(w(t))\widehat{w}_2(t)\big)dt+\widehat{\sigma}_1
\cdot \widehat{\sigma}_2.$$
Note that the Riemannian metric $\langle\cdot,\cdot \rangle$ on $T\big(\mathcal{P}\times (0,\infty)\big)$ induces
up to equivalence the metric $d$ on $\mathcal{P} \times (0,\infty)$.

For any Hamiltonian $H \in C^\infty(M,\mathbb{R})$ the gradient of the Rabinowitz action functional 
$\nabla \mathcal{A}^H$ with respect to the metric $\langle\cdot, \cdot  \rangle$ 
at a point $y=(w,\sigma) \in \mathcal{P}\times (0,\infty)$ is implicitly defined by the condition
$$d\mathcal{A}^H(y) \widehat{y}=\big\langle \nabla \mathcal{A}^H(y),\widehat{y}\big\rangle, \quad \forall\,\,
\widehat{y}=(\widehat{w},\widehat{\sigma}) \in T_y\big(\mathcal{P}\times (0,\infty)\big).$$
From (\ref{rabdiff}) we infer that
\begin{eqnarray*}
d\mathcal{A}^H(y) \widehat{y}&=&\int_0^1 \omega\big(\widehat{w},\partial_t w-\sigma X_H(w)\big)-\widehat{\sigma}
\int_0^1 H(w)dt\\
&=&-\int_0^1 g\big(\widehat{w},J(w)(\partial_t w-\sigma X_H(w)\big)-\widehat{\sigma}
\int_0^1 H(w)dt.
\end{eqnarray*}
Hence with respect to the splitting $T_y\big(\mathcal{P}\times (0,\infty)\big)=T_w \mathcal{P}\times \mathbb{R}$ the
gradient of the Rabinowitz action functional becomes
$$\nabla \mathcal{A}^H(y)=-\left(\begin{array}{c}
J(w)(\partial_t w-\sigma X_H(w))\\
\int_0^1 H(w)dt
\end{array}\right).$$
We abbreviate by
$$A_{\rho_0}(\Omega):=\overline{B_{2\rho_0/3}(\Omega)} \setminus B_{\rho_0/3}(\Omega)=\big\{y \in \mathcal{P}\times (0,\infty): \rho_0/3 \leq d(y,\Omega) \leq 2\rho_0/3\big\}.$$
Using this notion we are now in position to state our next lemma.
\begin{lemma}\label{epsi}
There exists $\epsilon>0$ such that
$$||\nabla \mathcal{A}^{H_{\mu_\infty}}(y)||\geq \epsilon,\quad \forall\,\, y \in A_{\rho_0}(\Omega)$$
where the norm $||\cdot||$ is taken with respect to the metric $\langle \cdot,\cdot \rangle$.
\end{lemma}
\begin{proof}
We argue by contradiction and assume that there exists a sequence 
$$y_\nu=(w_\nu,\sigma_\nu) \in A_{\rho_0}(\Omega)$$
with the property that
$$||\nabla \mathcal{A}^{H_{\mu_\infty}}(y_\nu)|| \leq \tfrac{1}{\nu},\quad \nu \in \mathbb{N}.$$
To derive a contradiction we first show the following Claim.
\\ \\
\textbf{Claim\,1:} \emph{There exists $\nu_0 \in \mathbb{N}$ with the property that for every $\nu \geq \nu_0$ 
the image of the path $w_\nu$ is completely contained in the compact subset $H_{\mu_\infty}^{-1}([-1/c,1/c])$ of $M$.}
\\ \\
In order to prove the Claim we first observe that since $H_{\mu_\infty}^{-1}([-1/c,1/c]) \subset M$ is compact there
exists  $\kappa>0$ with the property that 
$$|\nabla H_{\mu_\infty}(x)| \leq \kappa, \quad x \in H_{\mu_\infty}^{-1}([-1/c,1/c]).$$
Here the gradient and the norm are taken with respect to the Riemannian metric $g$ on $TM$. By property (b) in the choice
of $\rho_0$ we know that there exists $\tau_0 \in [0,1]$ with the property that
$$
|H_{\mu_\infty}(w_\nu(\tau_0))| \leq \tfrac{1}{2c}.$$
Now suppose that the path $w_\nu$ is not contained entirely in the set $H_{\mu_\infty}^{-1}([-1/c,1/c])$. This means that there
exists $\tau_1 \in [0,1]$ with the property that 
$$|H_{\mu_\infty}(w_\nu(\tau_1))| \geq \tfrac{1}{c}.$$
We estimate
\begin{eqnarray*}
\frac{1}{2c}&\leq& \big|H_{\mu_\infty}(w_\nu(\tau_1))-H_{\mu_\infty}(w_\nu(\tau_0))\big|\\
&=&\bigg|\int_{\tau_0}^{\tau_1} \frac{d}{dt}H_{\mu_\infty}(w_\nu(t))dt\bigg|\\
&=&\bigg|\int_{\tau_0}^{\tau_1} dH_{\mu_\infty}(w_\nu(t)) \partial_t w_\nu(t)dt\bigg|\\
&=&\bigg|\int_{\tau_0}^{\tau_1} dH_{\mu_\infty}(w_\nu(t)) \big(\partial_t w_\nu(t)-\sigma_\nu X_{H_{\mu_\infty}}(w_\nu(t))\big)dt\bigg|\\
&=&\bigg|\int_{\tau_0}^{\tau_1} g\big(\nabla H_{\mu_\infty}(w_\nu(t)),\partial_t w_\nu(t)-\sigma_\nu X_{H_{\mu_\infty}}(w_\nu(t))\big)dt\bigg|\\
&\leq&\int_{\tau_0}^{\tau_1} \big|\nabla H_{\mu_\infty}(w_\nu(t))\big| \cdot \big|\partial_t w_\nu(t)-\sigma_\nu X_{H_{\mu_\infty}}(w_\nu(t))\big|dt\\
&\leq&\kappa \int_{\tau_0}^{\tau_1}\big|\partial_t w_\nu(t)-\sigma_\nu X_{H_{\mu_\infty}}(w_\nu(t))\big|dt\\
&\leq&\kappa \big|\big|\partial_t w_\nu(t)-\sigma_\nu X_{H_{\mu_\infty}}(w_\nu(t))\big|\big|_1\\
&\leq&\kappa \big|\big|\partial_t w_\nu(t)-\sigma_\nu X_{H_{\mu_\infty}}(w_\nu(t))\big|\big|_2\\
&\leq& \kappa ||\nabla \mathcal{A}^{H_{\mu_\infty}}(y_\nu)||\\
&\leq& \frac{\kappa}{\nu}.
\end{eqnarray*} 
In view of this estimate the assertion of Claim\,1 follows by choosing
$$\nu_0 \geq 2 \kappa c.$$
\textbf{Claim\,2:} \emph{The sequence $w_\nu$ is equicontinuous.}
\\ \\
Let $\nu_0 \in \mathbb{N}$ be as in Claim\,1. It suffices to check equicontinuity for $\nu \geq \nu_0$. Note that in view of
Claim\,1 it follows that 
$$|X_{H_{\mu_\infty}}(w_\nu(t))|=|\nabla H_{\mu_\infty}(w_\nu(t))| \leq \kappa,\quad t \in [0,1].$$
Note further that
\begin{equation}\label{sbound}
\sigma_\nu \leq \max\big\{\sigma:(w,\sigma) \in \Omega\big\}+\frac{2 \rho_0}{3}=:\kappa_1.
\end{equation}
Given $0 \leq t_0<t_1 \leq 1$ we estimate
\begin{eqnarray*}
|w_\nu(t_1)-w_\nu(t_0)|&\leq&\int_{t_0}^{t_1}|\partial_t w_\nu(t)|dt\\
&\leq&\int_{t_0}^{t_1}|\partial_t w_\nu(t)-\sigma_\nu X_{H_{\mu_\infty}}(w_\nu(t))|dt\\
& &+\int_{t_0}^{t_1}|\sigma_\nu X_{H_{\mu_\infty}}(w_\nu(t))|dt\\
&\leq&\sqrt{t_1-t_0}\sqrt{\int_{t_0}^{t_1}|\partial_t w_\nu(t)-\sigma_\nu X_{H_{\mu_\infty}}(w_\nu(t))|^2dt}\\
& &+(t_1-t_0)\kappa_1 \kappa\\
&\leq&\sqrt{t_1-t_0}\big|\big|\partial_t w_\nu(t)-\sigma_\nu X_{H_{\mu_\infty}}(w_\nu(t))\big|\big|_2+(t_1-t_0)\kappa_1 \kappa\\
&\leq&\sqrt{t_1-t_0}\big|\big|\nabla \mathcal{A}^{H_{\mu_\infty}}(y_\nu\big|\big|+(t_1-t_0)\kappa_1 \kappa\\
&\leq&\frac{\sqrt{t_1-t_0}}{\nu}+(t_1-t_0)\kappa_1 \kappa\\
&\leq&\frac{\sqrt{t_1-t_0}}{\nu_0}+(t_1-t_0)\kappa_1 \kappa.
\end{eqnarray*}
This proves that the sequence $w_\nu$ is equicontinuous and establishes the truth of Claim\,2.
\\ \\
Claim\,1 and Claim\,2 together with the uniform bound in (\ref{sbound}) allow us to apply the Theorem of Arzela-Ascoli. That means that there exists a subsequence $\nu_j$ and 
$$y \in A_{\rho_0}(\Omega)$$
such that
$$\lim_{j \to \infty}y_{\nu_j}=y.$$
In particular,
$$\nabla \mathcal{A}^{H_{\mu_\infty}}(y)=0$$
or equivalently
$$d\mathcal{A}^{H_{\mu_\infty}}(y)=0.$$
Hence $y$ is a critical point of the Rabinowitz action functional and therefore by Proposition~\ref{crit} it follows that $y$ is a chord lying in
$$A_{\rho_0}(\Omega) \subset B_{\rho_0}(\Omega) \setminus \Omega.$$
This contradicts Lemma~\ref{ball}. The proof is complete. 
\end{proof} 

Because by \eqref{bou} the quantities $|H'_\mu|$ and $|X_{H'_\mu}|$ are uniformly bounded, we obtain the following result.
\begin{lemma}\label{bou2}
Let $\epsilon$ be as in Lemma~\ref{epsi}. Then there exists $\delta_0>0$ with the property that 
$$||\nabla \mathcal{A}^{H_\mu}(y)|| \geq \frac{\epsilon}{2},\qquad \mu \in [\mu_\infty-\delta_0,\mu_\infty+\delta_0],\quad
y \in A_{\rho_0}(\Omega).$$
\end{lemma}
We next introduce moduli spaces of gradient flow lines of time dependent Rabinowitz action functionals. Choose a cutoff function
$$
\beta \in C^\infty(\mathbb{R},[0,1])
$$
satisfying
$$
\beta'(s) \geq 0,\,\,s \leq 0,\qquad \beta'(s) \leq 0,\,\,s \geq 0,\qquad \beta(s)=0,\,\,|s| \geq T
$$
for some  $T=T(\beta)>0$. For 
$$0 \leq \mu_0<\mu_\infty<\mu_1 \leq 1$$ 
such that $\mu_0$ is close enough to $\mu_\infty$ in the sense that 
$$
(v_{\mu_0},\tau_{\mu_0}) \in B_{\rho_0}(\Omega)
$$
we introduce the time dependent family of Rabinowitz action functionals
$$\mathcal{A}_{\beta,\mu_0,\mu_1} \colon \mathcal{P}\times (0,\infty)\times \mathbb{R}\to \mathbb{R}, \quad 
(y,s) \mapsto \mathcal{A}^{H_{\mu_0+\beta(s)(\mu_1-\mu_0)}}(y).$$
We introduce the moduli space of gradient flow lines of $\mathcal{A}_{\beta,\mu_0,\mu_1}$ contained in
$B_{\rho_0}(\Omega)$ doubly asymptotic to
the chord $(v_{\mu_0},\tau_{\mu_0})$, i.e., 
$$\mathcal{M}(\beta,\mu_0,\mu_1) \subset C^\infty\big(\mathbb{R},B_{\rho_0}(\Omega)\big)$$
consists of all $y \colon \mathbb{R} \to B_{\rho_0}(\Omega)$ satisfying 
$$\partial_s y(s)+\nabla \mathcal{A}_{\beta,\mu_0,\mu_1}\big(y(s),s\big)=0,\qquad \lim_{s \to \pm \infty}y(s)=(v_{\mu_0},\tau_{\mu_0}).$$
\begin{prop}\label{nichtaus1}
There exists $\delta_1>0$ with the following property. If 
$$\mu_\infty-\delta_1\leq \mu_0<\mu_\infty<\mu_1<\mu_\infty+\delta_1$$
and $y \in \mathcal{M}(\beta,\mu_0,\mu_1)$ then
$$d\big(y(s),\Omega\big) \leq \frac{2\rho_0}{3},\quad s \in \mathbb{R}.$$
\end{prop}
\begin{proof}
Let $y=(w,\sigma) \in \mathcal{M}(\beta,\mu_0,\mu_1)$. We estimate the energy of $y$. For that purpose we abbreviate by
$\mathcal{A}'_{\beta,\mu_0,\mu_1}$ the derivative with respect to the second variable. We further introduce the constant
$$\kappa:=\max\{r: (u,r) \in \Omega\}+\rho_0.$$
Note that
$$\sigma(s) \leq \kappa,\quad \forall\,\,s \in \mathbb{R}.$$
Using that $y$ is doubly asymptotic to the same critical point of the Rabinowitz action functional $\mathcal{A}^{H_{\mu_0}}$ we obtain
\begin{eqnarray*}
0&=&\lim_{s \to \infty}\mathcal{A}_{\beta,\mu_0,\mu_1}\big(y(s),s\big)-\lim_{s\to-\infty}\mathcal{A}_{\beta,\mu_0,\mu_1}\big(y(s),s\big)\\
&=&\int_{-\infty}^\infty \frac{d}{ds}\mathcal{A}_{\beta,\mu_0,\mu_1}\big(y(s),s\big)ds\\
&=&\int_{-\infty}^\infty \mathcal{A}'_{\beta,\mu_0,\mu_1}\big(y(s),s\big)ds+\int_{-\infty}^\infty d\mathcal{A}_{\beta,\mu_0,\mu_1}\big(y(s),s\big)\partial_s y(s)ds\\
&=&(\mu_0-\mu_1)\int_{-T}^T \sigma(s)\beta'(s) H'_{\mu_0+\beta(s)(\mu_1-\mu_0)}(w(s))ds\\
& &-\int_{-\infty}^\infty
\big|\big|\nabla \mathcal{A}_{\beta,\mu_0,\mu_1}\big(y(s),s\big)\big|\big|^2ds.
\end{eqnarray*}
This gives the inequality
\begin{eqnarray}\label{en1}
\int_{-\infty}^\infty
\big|\big|\nabla \mathcal{A}_{\beta,\mu_0,\mu_1}\big(y(s),s\big)\big|\big|^2ds
&\leq& (\mu_1-\mu_0)\kappa c\int_{-T}^T |\beta'(s)| ds\\ \nonumber
&\leq& 2(\mu_1-\mu_0)\kappa c.
\end{eqnarray}
Assume that $\delta_1 \leq \delta_0$ where $\delta_0$ is as in Lemma~\ref{bou2} and where $\delta_1$ is so small such that for every
$\mu \in [\mu_\infty-\delta_1,\mu_\infty)$ it holds that 
$$(v_\mu,\tau_\mu) \in B_{\rho_0/3}(\Omega).$$
Choose $\mu_1 \in [\mu_\infty,\mu_\infty+\delta_1)$ and suppose that $y \in \mathcal{M}(\beta,\mu_0,\mu_1)$ which
has the property that $y$ is not contained for all times in the closed ball of radius $2\rho_0/3$ around $\Omega$.
In view of the asymptotic behavior of $y$ this implies that there exist $s_0<s_1$ satisfying
$$d\big(y(s_0),\Omega\big)=\frac{\rho_0}{3},\quad d\big(y(s_1\big),\Omega)=\frac{2\rho_0}{3},\quad
\frac{\rho_0}{3} \leq d\big(y(s),\Omega\big)\leq \frac{2\rho_0}{3},\,\,s \in [s_0,s_1].$$
We estimate using Lemma~\ref{bou2}
\begin{eqnarray}\label{en2}
\int_{-\infty}^\infty
\big|\big|\nabla \mathcal{A}_{\beta,\mu_0,\mu_1}\big(y(s),s\big)\big|\big|^2ds
&\geq&\int_{s_0}^{s_1}
\big|\big|\nabla \mathcal{A}_{\beta,\mu_0,\mu_1}\big(y(s),s\big)\big|\big|^2ds\\ \nonumber
&\geq&\frac{\epsilon}{2}\int_{s_0}^{s_1}
\big|\big|\nabla \mathcal{A}_{\beta,\mu_0,\mu_1}\big(y(s),s\big)\big|\big|ds\\ \nonumber
&=&\frac{\epsilon}{2}\int_{s_0}^{s_1}
\big|\big|\partial_s y(s)\big|\big|ds\\ \nonumber
&\geq&\frac{\epsilon}{2}d\big(y(s_1),y(s_0)\big)\\ \nonumber
&\geq&\frac{\epsilon \rho_0}{6}.
\end{eqnarray}
Combining the inequalities \eqref{en1} and \eqref{en2} we obtain the inequality
$$\frac{\epsilon \rho_0}{6} \leq 2(\mu_1-\mu_0)\kappa c \leq 4 \delta_1 \kappa c.$$
Hence by choosing
$$\delta_1<\frac{\epsilon \rho_0}{24\kappa c}$$
this inequality cannot be true anymore and therefore $y$ has to stay in the closed $2\rho_0/3$-ball around $\Omega$.
This finishes the proof of the proposition. 
\end{proof}

Given a cutoff function $\beta$ and $0 \leq \mu_0 <\mu_\infty<\mu_1 \leq 1$ we introduce the following
subset of 
$$\mathcal{M}_0(\beta,\mu_0,\mu_1) \subset \mathcal{M}(\beta,\mu_0,\mu_1)$$
consisting of all $y=(w,\sigma) \in \mathcal{M}(\beta,\mu_0,\mu_1)$ satisfying 
$$|H_{\mu_0+\beta(s)(\mu_1-\mu_0)}(w(s,t))|
< \frac{1}{c},\quad \forall\,\,s \in \mathbb{R},\,\,t \in [0,1].$$
\begin{prop}\label{bubble}
Assume that $y_\nu=(w_\nu,\sigma_\nu)$ is a sequence in $\mathcal{M}_0(\beta,\mu_0,\mu_1)$. Then there exists a subsequence 
$\nu_j$ and
a gradient flow line $y$ of the time-dependent Rabinowitz action functional $\mathcal{A}_{\beta,\mu_0,\mu_1}$, i.e.,
a solution of
$$\partial_s y(s)+\nabla \mathcal{A}_{\beta,\mu_0,\mu_1}\big(y(s),s\big),\quad s \in \mathbb{R}$$
such that $y_{\nu_j}$ converges in the $C^\infty_{\mathrm{loc}}\big(\mathbb{R}\times [0,1],M\big) \times
C^\infty_{\mathrm{loc}}\big(\mathbb{R},(0,\infty)\big)$-topology to $y$.
\end{prop}
\begin{proof}
Note that $\sigma_\nu$ is uniformly bounded and therefore $w_\nu$ satisfies a perturbed Cauchy-Riemann
equation of bounded energy with bounded perturbation. Moreover because $y_\nu$ lies in $\mathcal{M}_0(\beta,\mu_0,\mu_1)$ its image is contained in the compact subset
$$\bigcup_{\mu \in [0,1]} H^{-1}_\mu\big(\big[-\tfrac{1}{c},\tfrac{1}{c}\big]\big)\subset M.$$
Because $\omega=d\lambda$ is exact and $\lambda$ vanishes on the Lagrangians $L_0$ and $L_1$, there is neither bubbling
of holomorphic spheres nor holomorphic disks. The proposition follows. 
\end{proof}

\begin{prop}\label{nichtaus2}
There exists $\delta_2>0$ with the following property. Assume that 
$$\mu_0 \in [\mu_\infty-\delta_2,\mu_\infty),\quad \mu_1 \in (\mu_\infty,\mu_\infty+\delta_2]$$
Then for every $y=(w,\sigma) \in \mathcal{M}_0(\beta,\mu_0,\mu_1)$ it holds that
$$|H_{\mu_0+\beta(s)(\mu_1-\mu_0)}(w(s,t))|
< \frac{1}{2c},\quad \forall\,\,s \in \mathbb{R},\,\,t \in [0,1].$$
\end{prop}
\begin{proof}
We argue by contradiction and assume instead that there exists sequences $\mu_0^\nu <\mu_\infty$
and $\mu_1^\nu>\mu_\infty$ satisfying
\begin{equation}\label{limi}
\lim_{\nu\to \infty}\mu_0^\nu=\mu_\infty=\lim_{\nu \to \infty}\mu_1^\nu
\end{equation}
such that there exists $y_\nu \in \mathcal{M}_0(\beta,\mu_0^\nu,\mu_1^\nu)$, $s_\nu \in \mathbb{R}$ and $t_\nu \in [0,1]$ such that
\begin{equation}\label{aus}
|H_{\mu^\nu_0+\beta(s_\nu)(\mu^\nu_1-\mu^\nu_0)}(w(s_\nu,t_\nu))|
\geq \frac{1}{2c}.
\end{equation}
We consider the sequence of time-shifted gradient flow lines
$$(s_\nu)_*y_\nu(s):=y_\nu(s+s_\nu),\quad s \in \mathbb{R}.$$
By the arguments in the proof of Proposition~\ref{bubble} and \eqref{limi} there exists a subsequence $\nu_j$ and
a gradient flow line $y$ of the Rabinowitz action functional
$$\mathcal{A}_{\beta,\mu_\infty,\mu_\infty}=\mathcal{A}^{H_{\mu_\infty}}$$
such that $(s_\nu)_* y_\nu$ converges to $y$. Note that in view of \eqref{limi} and the fact that $y_\nu$ is doubly
asymptotic to the same chord, the energy
$$
E\big((s_\nu)_*y_\nu\big)=E\big(y_\nu\big)=\int_{-\infty}^{\infty}\big|\big|\partial_s y_\nu(s)\big|\big|^2ds
$$
converges to zero as $\nu$ goes to infinity. This implies that $y=(w,\sigma)$ is just a constant gradient flow line of the
Rabinowitz action functional $\mathcal{A}^{H_\infty}$, i.e., $y$ is a critical point and therefore has to be
a chord in $\Omega$. Because the sequence $t_\nu$ is contained in the compact interval $[0,1]$ we can assume, maybe after going to a further subsequence, that there exists $t_\infty \in [0,1]$ such that
$$\lim_{j \to \infty}t_{\nu_j}=t_\infty.$$
Hence we infer from (\ref{aus}) that
$$|H_{\mu_\infty}(w(t_\infty))| \geq \frac{1}{2c}.$$
This contradicts the fact that the image of the chord $w$ is completely contained in $\Sigma_{\infty}=H^{-1}_{\mu_\infty}(0)$.
The proof of the Proposition is finished. 
\end{proof}

Choose $\delta_1>0$ as in Proposition~\ref{nichtaus1} and $\delta_2>0$ as in Proposition~\ref{nichtaus2} and set
$$\delta_3:=\min\{\delta_1,\delta_2\}.$$
\begin{prop}\label{fast}
Assume that $\mu_\infty-\delta_3<\mu_0<\mu_\infty<\mu_1<\mu_\infty+\delta_3$ and suppose that
\begin{equation}\label{nobreak}
B_{\rho_0}(\Omega) \cap \mathrm{crit}\mathcal{A}^{H_{\mu_0}}=\{(v_{\mu_0},\tau_{\mu_0})\}.
\end{equation}
Then $\mathrm{crit}\mathcal{A}^{H_{\mu_1}}\cap B_{\rho_0}(\Omega) \neq \emptyset$.
\end{prop}
\begin{proof}
Choose a smooth function $\gamma \in C^\infty(\mathbb{R},[0,1])$ satisfying 
$$\gamma' \geq 0,\quad \gamma(s) =0,\,\,s \leq -1,\quad \gamma(s)=1,\,\,s \geq 1.$$
Define a smooth one-parameter family of cutoff functions 
$$\beta_R \in C^\infty(\mathbb{R},[0,1]),\quad R \in [0,\infty)$$
as follows
$$\beta_R(s)=\left\{\begin{array}{cc}
R\gamma(2+s) & R \in [0,1],\,\,s\leq 0,\\
R\gamma(2-s) & R \in [0,1],\,\,s\geq 0,\\
\gamma(1+s+R) & R \geq 1,\,\,s \leq 0,\\
\gamma(1-s+R) & R \geq 1,\,\,s \geq 1.
\end{array}\right.$$
We now consider the one-parameter family of moduli spaces
$$\mathcal{M}_0^R:=\mathcal{M}_0(\beta_R,\mu_0,\mu_1),\quad R \in [0,\infty).$$
\textbf{Claim: } \emph{For every $R \in [0,\infty)$ the moduli space $\mathcal{M}_0^R$ is nonempty. }
\\ \\
In order to prove the Claim we consider the moduli space
$$\mathcal{N}_R:=\bigcup_{r \in [0,R]}\mathcal{M}_0^r\times \{r\}.$$
By assumption \eqref{nobreak}, gradient flow lines in $\mathcal{N}_R$ cannot break. Therefore Proposition~\ref{nichtaus1}
and Proposition~\ref{nichtaus2} combined with the proof of Proposition~\ref{bubble} show that the moduli space 
$\mathcal{N}_R$ is compact. After a small perturbation we can assume that it is a one-dimensional manifold with boundary, 
where the boundary is given by 
$$
\partial \mathcal{N}_R=\mathcal{M}_0^0 \times \{0\} \cup \mathcal{M}_0^R \times \{R\}.
$$
The moduli space $\mathcal{M}_0^0$ consists of gradient flow lines of the time-independent Rabinowitz action functional
$\mathcal{A}^{H_{\mu_0}}$ doubly asymptotic to the chord $(v_{\mu_0},\tau_{\mu_0})$. However, because of time-independence
such a gradient flow line has to be constant, i.e., is just given by the chord $(v_{\mu_0},\tau_{\mu_0})$. Hence
$\mathcal{M}_0^0$ consists just of a single point. Note further that, because $(v_{\mu_0},\tau_{\mu_0})$ is non-degenerate,
this boundary point of $\mathcal{N}_R$ is non-degenerate so that we actually do not need to perturb there to get a manifold
structure. However, a one dimensional manifold is a disjoint union of intervals and circles. In particular, the  number of
boundary points is even. This proves that $\mathcal{M}_0^R \neq \emptyset$ and establishes the truth of the Claim. 
\\ \\
Choose now a sequence $R_\nu$ converging to infinity. By the claim there exist gradient flow lines
$$y_\nu \in \mathcal{M}_0^{R_\nu}.$$ By Proposition~\ref{nichtaus1} and Proposition~\ref{nichtaus2} combined with the
proof of Proposition~\ref{bubble} there exists a subsequence $\nu_j$ and a gradient flow line $y=(w,\sigma)$ of the Rabinowitz action functional $\mathcal{A}^{H_{\mu_1}}$ whose image is contained in $B_{\rho_0}(\Omega)$ and moreover the image of
$w$ is contained in the compact subset $\bigcup_{\mu \in [0,1]}H^{-1}_\mu([-1/c,1/c])\subset M$ such that
$$\lim_{j \to \infty} y_{\nu_j}=y.$$
Now choose a sequence $s_\nu$ converging to infinity. Then there exists a subsequence $\nu_j$ such that the
sequence 
$$y(s_{\nu_j}) \in B_{\rho_0}(\Omega)$$
converges to a critical point of $\mathcal{A}^{H_{\mu_1}}$. This finishes the proof of the Proposition. 
\end{proof}

\begin{cor}\label{fastcor}
Under the assumptions of Proposition~\ref{fast} there exists $\mu_1>\mu_\infty$ such that for every
$\mu \in (\mu_\infty,\mu_1)$ it holds that $\mathrm{crit}\mathcal{A}^{H_\mu}\cap U \neq \emptyset$.
\end{cor}
\begin{proof}
We argue by contradiction. In this case we can assume by Proposition~\ref{fast} that there exists a sequence
$\mu_\nu>\mu_\infty$ converging to $\mu_\infty$ with the property that
$$\mathrm{crit}\mathcal{A}^{H_{\mu_\nu}}\cap B_{\rho_0}(\Omega) \neq \emptyset, \quad \mathrm{crit}\mathcal{A}^{H_{\mu_\nu}}\cap U = \emptyset.$$
Hence we choose
$$y_\nu \in \mathrm{crit}\mathcal{A}^{H_{\mu_\nu}}\cap B_{\rho_0}(\Omega),\quad y_\nu \notin U.$$
Because $\mu_\nu$ converges to $\mu_\infty$ there exists a subsequence $\nu_j$ and a chord $y \in \Omega$ such that
$$\lim_{j \to \infty}y_{\nu_j}=y.$$
However, in view of the continuity of the chord equation this implies that there exists $j_0 \in \mathbb{N}$ such that
$$y_{\nu_j} \in U, \quad j \geq j_0.$$
This contradiction proves the Corollary. 
\end{proof}

We are now in position to prove the main result of this section.

\begin{proof}[Proof of Theorem~\ref{mainthm2}:]
By assumption of the Theorem and by Corollary~\ref{fastcor} we conclude that 
for $\mu \in (\mu_\infty-\delta_3,\mu_\infty)$ it holds that
$$\#(B_{\rho_0}(\Omega) \cap \mathrm{crit}\mathcal{A}^{H_\mu})\geq 2.$$
But then the same argument as in the proof of Corollary~\ref{fastcor} implies that there exists $0<\delta \leq \delta_3$
with the property that 
$$\#(U \cap \mathrm{crit}\mathcal{A}^{H_\mu})\geq 2.$$
This finishes the proof of the Theorem. 
\end{proof}


\begin{thebibliography}{99}
\bibitem{albers-frauenfelder}
P.\,Albers, U.\,Frauenfelder, \emph{Leaf-wise intersections and Rabinowitz Floer homology}, J. Topol. Anal. 2 (2010), no. 1, 77--98. 

\bibitem{albers-frauenfelder-koert-paternain} P.\,Albers, U.\,Frauenfelder, O.\,van Koert, G.\,Paternain, \emph{Contact geometry of the restricted three-body problem}, Comm.\,Pure Appl.\,Math. \textbf{65} (2012), no.\,2, 229--263.

\bibitem{belbruno-frauenfelder-vankoert}
E.~Belbruno, U.~Frauenfelder, O.~van Koert, \emph{A Family of Periodic Orbits in the Three-Dimensional Lunar Problem}, arXiv:1811.05897

\bibitem{cho-kim}
W.~Cho, G.~Kim, \emph{The circular spatial restricted 3-body problem}, arXiv:1810.05796

\bibitem{cieliebak-frauenfelder}
K.~Cieliebak, U.~Frauenfelder, \emph{A Floer homology for exact contact embeddings} Pacific J. Math. 239 (2009), no. 2, 251--316.

\bibitem{frauenfelder-vankoert}
U.~Frauenfelder, O.~van Koert, \emph{The Restricted Three-Body Problem and Holomorphic Curves}, Springer Verlag, Pathways in Mathematics, 2018.

\bibitem{kang}
J.~Kang, \emph{Some remarks on symmetric periodic orbits in the restricted three-body problem}, Discrete Contin. Dyn. Syst. 34 (2014), no. 12, 5229--5245. 

\bibitem{merry}
W.~Merry, \emph{Lagrangian Rabinowitz Floer homology and twisted cotangent bundles} Geom. Dedicata 171 (2014), 345--386.

\bibitem{rabinowitz} P.~Rabinowitz, \emph{Periodic solutions of Hamiltonian systems}, Comm. Pure Appl. Math. 31 (1978), no. 2, 157--184. 

\bibitem{salomao}
P.~Salom\~ao, \emph{Convex energy levels of Hamiltonian systems}, Qual. Theory Dyn. Syst. 4 (2003), no. 2, 439--457 (2004). 
\end{thebibliography}
\end{document}